\documentclass[11pt,reqno]{amsart}
\usepackage{geometry}
\usepackage{stmaryrd}
\usepackage{amsfonts,amsmath,amssymb}
\usepackage{enumerate}
\usepackage{mathabx}
\usepackage{mathrsfs}
\usepackage[latin1]{inputenc}
\usepackage[T1]{fontenc}
\usepackage[usenames,dvipsnames]{pstricks}
\usepackage{appendix}
\newtheorem{theorem}{Theorem}[section]

\newtheorem{lemma}[theorem]{Lemma}

\newtheorem{proposition}[theorem]{Proposition}

\newcommand{\R}{\mathbb{R}}
\newcommand{\C}{\mathbb{C}}

\newcommand{\N}{\mathbb{N}}

\newcommand{\cQ}{\mathcal{Q}}

\newcommand{\im}{\mathrm{Im}}
\newcommand{\norm}[1]{\|#1\|}
\newcommand{\abs}[1]{|#1|}

\usepackage{graphicx}

\newcommand{\bsd}{\textrm{BSD}}

\usepackage{mathrsfs}
\usepackage[latin1]{inputenc}
\usepackage[T1]{fontenc}
\usepackage{color,xspace}
\usepackage[usenames,dvipsnames]{pstricks}
\usepackage{times}
\allowdisplaybreaks
\usepackage{subcaption}
\usepackage{hyperref}
\hypersetup{
    colorlinks=true,
    linkcolor=blue,
    filecolor=blue,
    urlcolor=cyan,
    pdftitle={Overleaf Example},
    pdfpagemode=A4
    }

\begin{document}
\title[Stable determination of unbounded potential by asymptotic boundary spectral data]
{Stable determination of unbounded potential by asymptotic boundary spectral data}
%
\author[Y. Kian and \'E. Soccorsi]{Yavar Kian and \'Eric Soccorsi}
\address{Aix-Marseille Univ, Université de Toulon, CNRS, CPT, Marseille, France}

\email{yavar.kian@univ-amu.fr}
 
\address{Aix-Marseille Univ, Université de Toulon, CNRS, CPT, Marseille, France}

\email{eric.soccorsi@univ-amu.fr}


\maketitle               

\begin{abstract}             
We consider the Dirichlet Laplacian $A_q=-\Delta+q$
in a bounded domain $\Omega \subset \mathbb{R}^d$, $d \ge 3$, with real-valued perturbation $q \in L^{\max(2,3 d \slash 5)}(\Omega)$. We examine the stability issue in the inverse problem of determining the electric potential $q$
from the asymptotic behavior of the eigenvalues of $A_q$. Assuming that the boundary measurement of the normal derivative of the eigenfunctions is a square summable sequence in $L^2(\partial \Omega)$, we prove that $q$ can be H\"older stably retrieved through knowledge of the asymptotics of the eigenvalues. \end{abstract}

\section{Introduction}
\subsection{Settings}
Let $\Omega \subset \mathbb{R}^d$, $d \geq 3$, be a bounded domain with $\mathcal C^2$ boundary $\Gamma:=\partial\Omega$. 
We consider the perturbed Dirichlet Laplacian $A_q:=-\Delta +q$ in $L^2(\Omega)$, where $q$ is taken within the class
$$ \cQ_{c_0}(M) := \left\{ q \in L^{\max(2,3 d \slash 5)}(\Omega,\R)\ \mbox{s. t.}\ \norm{q}_{L^{\max(2,3 d \slash 5}(\Omega)} \leq M\ \mbox{and}\ q(x) \geq -c_0,\ x \in \Omega \right\}, $$
associated with two {\it a priori} fixed positive constants $c_0$ and $M$. More precisely, $A_q$ is the self-adjoint operator in $L^2(\Omega)$, generated by the closed Hermitian form
$$
a_q(u,v) := \int_\Omega \left( \nabla u \cdot \nabla \overline{v} + q u \overline{v} \right) dx,\ u, v \in D(a_q):=H_0^1(\Omega),
$$
see e.g. \cite[Appendix A]{BKMS}. By compact embedding of $H_0^1(\Omega)$ in $L^2(\Omega)$, the operator $A_q$ has a compact resolvent. Therefore, there exist a sequence of eigenfunctions $\phi_k \in D(A_q)=\{ u \in H_0^1(\Omega),\ (-\Delta+q) u \in L^2(\Omega) \}$ which form an orthonormal basis of $L^2(\Omega)$, and a sequence of eigenvalues
$$-c_0<\lambda_{1}\leq \lambda_{2}\leq \ldots \leq \lambda_{k} \leq \lambda_{k+1} \leq \ldots,$$
repeated with the multiplicity, satisfying
$\lim_{k \to +\infty} \lambda_k=+\infty$ and
$$
A_q \phi_k=\lambda_k \phi_k,\ k \geq 1.
$$
For all $k \geq 1$, we put $\psi_k:=(\partial_\nu \phi_k)_{| \Gamma}$, where $\nu$ denotes the outward pointing unit normal vector to $\Gamma$.

In the present article we study the stability issue in the inverse problem of determining the potential $q$ from 
the asymptotic behavior with respect to $k$, of the so-called boundary spectral data, 
$$\bsd(A_q):=\{(\lambda_k,\psi_k),\ k \geq 1\}. $$

\subsection{State of the art}
The study of mathematical inverse spectral problems has a long story which dates back to, at least, 1929 and Ambarsumian's pionneering article \cite{A}, where the author proved that real-valued potential $q$ of the Sturm-Liouville operator $A_q=-\partial_x^2+q$ acting in $L^2(0,2\pi)$ is zero if and only if the spectrum of the periodic realization of $A_q$ equals $\{k^2,\ k \in \N\}$. However the first breakthrough results in this field appeared between 1945 and 1951, when Borg \cite{Bo}, Levinson \cite{L}, and Gel'fand and Levitan \cite{GL} identified the electric potential of $A_q$ through knowledge of the spectrum and additional spectral data.

Almost 40 years later, in \cite{NSU}, Nachman, Sylvester and Uhlmann extended Gel'fand and Levitan's result to multidimensional Schrödinger operators. They showed that full knowledge of the boundary spectral data uniquely determines the electric potential. 
Later on, in \cite{I}, Isozaki showed that the identification result of Nachman, Sylvester and Uhlmann is still valid when finitely many eigenpairs 
$(\lambda_n,\psi_n)$ remain unknown. This result, which is often referred to as the incomplete Borg-Levinson theorem, was extended by Canuto and Kavian to the case of Schrödinger operators in the divergence form $-\rho \nabla \cdot c \nabla  + q$, when the conductivity $c$ and the density $\rho$ satisfy $\min(c,\rho) \ge \epsilon$ for some $\epsilon>0$. More precisely, they showed in \cite{CK1,CK2} that the boundary spectral data uniquely determine two out the three functions $(c,\rho,q)$.

In 2013, further downsizing the data, Choulli and Stefanov established in \cite{CS} that the electric potential can be retrieved from asymptotic knowledge of the boundary spectral data. Namely, the authors proved that two potentials are equal whenever their boundary spectral data are sufficiently close asymptotically. This result was improved by \cite{KKS,S} upon weakening the condition imposed on the asymptotics of the boundary spectral data.

The stability issue for the inverse problem of determining the electric potential from the full boundary spectral data was first treated in by Alessandrini, Sylvester and Sun in \cite{ASS}. Their result was adapted to local Neumann data in \cite{BCY} and to asymptotic boundary spectral data in \cite {CS,KKS,S}. 

All the above results were obtained for bounded electric potentials. As for the mathematical literature on inverse spectral problems with singular potentials, it seems to be quite sparse. The first result on this topic was published in \cite{PS} by P\"avarinta and Serov, who showed that knowledge of the full boundary spectral data uniquely determines the electric potential in $L^p(\Omega, \mathbb{R})$, provided that $p >d/2$. Later on, in \cite{P}, Pohjola showed unique determination of $q \in L^{d/2}(\Omega,\mathbb{R})$ from either full boundary spectral data when $d = 3$ or incomplete boundary spectral data when $d \geq 4$, and of $q \in L^{p}(\Omega,\mathbb{R})$ with $p >d/2$ and $d = 3$, from incomplete boundary spectral data. More recently, it was proved in \cite{BKMS} that $q \in L^{\max(2, 3 d \slash 5)}(\Omega)$ is uniquely determined by the asymptotic boundary spectral data. The three above mentioned papers are only concerned with the uniqueness problem and as far as we know, there is no result available in the mathematical literature dealing with the stability issue in the inverse problem of determining a singular potential from knowledge of the boundary spectral data.


\subsection{Stability estimate}

The main achievement of this article is the following stability estimate. 

\begin{theorem}
\label{t1}
For $j=1,2$, let $q_j\in \cQ_{c_0}(M)$, where $c_0>0$ and $M>0$ are fixed, and denote by $\{(\lambda_{j,k},\psi_{j,k}),\ k \geq 1 \}$ the boundary spectral data of the operator $A_{q_j}$. Assume that
\begin{equation}
\label{t1a}  \sum_{k=1}^{+ \infty} \norm{\psi_{1,k}-\psi_{2,k}}_{L^2(\Gamma)}^2 <+\infty.
\end{equation}
Then, there exists a positive constant $C$, depending only on $\Omega$, $c_0$ and $M$, such that we have
\begin{equation}
\label{se}
\norm{q_1-q_2}_{H^{-1}(\Omega)} \leq C \left( \limsup_{k \to +\infty} \abs{\lambda_{1,k}-\lambda_{2,k}} \right)^{\frac{1}{d+2}}.
\end{equation}
\end{theorem}

Theorem \ref{t1} is the main novelty of this work, as we are not aware of any stability result available in the mathematical literature, dealing with the determination of singular potentials by asymptotic boundary spectral data. It is proved in \cite{KKS,S} that the potential depends continuously on asymptotic boundary spectral data but this is for bounded potentials only.

Evidently, \eqref{se} yields unique determination of $q \in \mathcal{Q}_{c_0}(M)$ by asymptotic knowledge of the spectrum of $A_q$, in the sense that the following implication
$$ \limsup_{k \to +\infty} \abs{\lambda_{1,k}-\lambda_{2,k}}=0 \Longrightarrow q_1=q_2, $$
which was already established in \cite[Theorem 1.1]{BKMS}, holds under the conditions of Theorem \ref{t1}.

Notice that the stability inequality \eqref{se} involves only the asymptotic distance between the eigenvalues of the operators $A_{q_j}$, $j=1,2$, and does not need explicitly any quantitative information about $\norm{\psi_{1,k}-\psi_{2,k}}_{L^2(\Gamma)}$. This might seem a little bit surprising at first sight since distinct iso-spectral potentials $q_1$ and $q_2$ can be built on certain domains $\Omega$. But for such potentials the condition \eqref{t1a} is not fulfilled as one has $\sum_{n=1}^{+\infty} \norm{\psi_{1,k}-\psi_{2,k}}_{L^2(\Gamma)}^2=+\infty$.

In the same spirit, it is worth mentioning that Theorem \ref{t1} does not assume that $q_1-q_2 \in L^\infty(\Omega)$. As a matter of fact it seems very likely that we have $\limsup_{k \to +\infty} \abs{\lambda_{1,k}-\lambda_{2,k}}=+\infty$, in which case \eqref{se} is trivially satisfied, when $q_1-q_2$ is unbounded. However, we are not able to prove rigorously this claim (neither can we guarantee that it is true) although the converse implication is obvious.

\subsection{Structure of the article}
The paper is organized as follows. In Section \ref{sec-Isozaki} we extend the celebrated Isozaki's representation formula, which was initially established for bounded potentials, to the case of unbounded potentials. In Section \ref{sec-proof} we prove Theorem \ref{t1} by mean of Isozaki's formula. Finally, in the appendix, we collect several technical results that are needed for the proof of Theorem \ref{t1}.

\section{Isozaki's asymptotic formula}
\label{sec-Isozaki}
The goal of this section, which is inspired from \cite[Section 3]{BKMS}, is to relate the Fourier transform of $q_1-q_2$ to $\bsd(q_j)$, $j=1,2$, by triggering the boundary value problem \eqref{1} with suitable Dirichlet  
boundary data $f$. This approach was introduced by Isozaki in \cite{I}, and
since then it has been successfully applied by numerous authors to the analysis of 
multidimensional inverse spectral problems, see e.g., \cite{BCFKS,CS,KKS,Ki1,P,S}.

\subsubsection{A sufficiently rich set of test functions}
Let $\xi \in \mathbb{R}^d$ and set $\lambda_\tau^\pm = (\tau \pm i)^2$ for all $\tau \ge \abs{\xi}$. We seek two functions $f_\tau^\pm$ such that
\begin{equation}
\label{3.36}
( -\Delta -\lambda_{\tau}^\pm)f_\tau^\pm=0\ \mbox{in}\ \Omega
\end{equation}
and satisfying
\begin{equation}
\label{2}
\lim_{\tau \rightarrow +\infty} f_\tau^+(x) \overline{f_\tau^-(x)}=e^{-i\xi \cdot x},\; x\in \Omega,
\end{equation}
\begin{equation}
\label{3000}
    \sup_{\tau \ge \abs{\xi}} \norm{f_\tau^\pm}_{L^\infty(\Omega)} <+\infty.
\end{equation}
Pick $\eta \in \mathbb{S}^{n-1}$ such that $\xi \cdot \eta=0$, and for $\tau  \ge \abs{\xi}$, put
$$\beta_\tau := \sqrt{1-\frac{|\xi|^2}{4\tau^2}}\ \mbox{and}\ \eta_\tau^\pm := \beta_\tau \eta \mp \frac{\xi}{2\tau}$$ 
in such a way that $\abs{\eta_\tau^\pm}=1$. Then, it is apparent that
$$ f_\tau^\pm (x) := e^{i(\tau\pm i)\eta^{\pm}_{\tau} \cdot x},\; x\in \Omega,$$
fulfill \eqref{3.36} and \eqref{2}. Moreover, we have
$\abs{f_\tau^\pm (x)} \leq e^{\abs{x}}$ for all $x \in \overline{\Omega}$,  and hence
\begin{equation}
\label{3001}
\norm{f_\tau^\pm}_{L^r(X)} \leq \abs{X}^{1/r}  \sup_{x \in \overline{\Omega}} e^{\abs{x}}:=C_{r,X},\ X=\Omega,\Gamma,
\end{equation}
whenever $r \in [2,+\infty)$ or $r=+\infty$. More specifically, we notice that \eqref{3001} with $(r,X)=(+\infty,\Omega)$ yields \eqref{3000}.

Let $q \in \cQ_{c_0}(M)$, where $c_0>0$ and $M>0$ are fixed. Then, for all $\tau \geq \abs{\xi}$ we have $q f_\tau^\pm \in L^2(\Omega)$ by \eqref{3001}, and the estimate
$\norm{q f_\tau^\pm}_{L^2(\Omega)} \leq \norm{q}_{L^{2}(\Omega)} \norm{f_\tau^\pm}_{L^\infty(\Omega)} \leq M C_{+\infty,\Omega}$ when $d = 3$,
whereas
$\norm{q f_\tau^\pm}_{L^2(\Omega)} \leq \norm{q}_{L^{3 d \slash 5}(\Omega)} \norm{f_\tau^\pm}_{L^\frac{d}{d-2}(\Omega)} 
\leq M C_{\frac{d}{d-2},\Omega}$ when $d \geq 4$. Therefore, we have
\begin{equation}
\label{tm1}
\norm{q f_\tau^\pm}_{L^2(\Omega)} \leq C,
\end{equation}
for some positive constant $C$ which is independent of $\tau$. 

\subsubsection{Triggering the system with $f_\tau^\pm$}

For $j=1,2$, let $q_j \in \cQ_{c_0}(M)$, $z\in \mathbb{C}\setminus [-c_0,+\infty)$, and denote by $u_{j,z}^{\pm}$ the 
$W^{2,p}(\Omega)$-solution to the boundary value problem
\begin{equation}
\label{tm0}
 \left\{
\begin{array}{ll}
( -\Delta +{{q_j}}-z)u = 0  &\text{in}\ \Omega,\\
u = f_{\tau}^\pm & \text{on}\ \Gamma.
 \end{array}
\right. 
\end{equation}
Since $( -\Delta +{{q_j}}-z)f^\pm_\tau =({{q_j}}+\lambda_{\tau}^\pm -z) f^\pm_\tau$ from \eqref{3.36}, the function
\begin{equation}
\label{vz1}
v_{j,z}^\pm := u_{j,z}^\pm-f_\tau^\pm
\end{equation}
solves
 $$
 \left\{
\begin{array}{ll}
( -\Delta +{{q_j}}-z) v = -( -\Delta +{{q_j}}-z)f^\pm_\tau  & \text{in}\ \Omega\\
v= 0 & \text{on}\ \Gamma,
 \end{array}
\right. 
$$
and consequently we have
\begin{equation}
\label{vz2}
v^\pm_{j,z}=-(A_{q_j}-z)^{-1}(q_j+\lambda^\pm_\tau-z)f_\tau^\pm. 
\end{equation}
In the special case where $z=\lambda_\tau^\pm$, the above identity reads $v_{j,\lambda_\tau^\pm}^\pm = -(A_{q_j}-\lambda_{\tau}^\pm)^{-1} (q_j f_\tau^\pm)$. Since $\im\ \lambda_\tau^\pm =\pm 2 \tau$, we infer from \eqref{tm1} that
\begin{equation}
\label{vz3}
\norm{v^\pm_{j,\lambda_\tau^\pm}}_{L^2(\Omega)} \leq C \tau^{-1},\ \tau \ge \abs{\xi},
\end{equation}
where we recall that the constant $C$ is independent of $\tau$. From this, \eqref{tm1} and the Cauchy-Schwarz inequality, it then follows that 
$\abs{\int_\Omega q_j v_{j,\lambda_\tau^+}^+ \overline{f_\tau^-} dx} \leq C^2 \tau^{-1}$, and consequently we have
\begin{equation}
\label{l4a}
\lim_{\tau\to+\infty} \int_\Omega q_j v_{j,\lambda_\tau^+}^+ \overline{f_\tau^-} dx=0,\ j=1,2.
\end{equation}
Armed with \eqref{l4a}, we are now in position to establish the Isozaki formula for the unbounded potentials $q_j$, $j=1,2$. 

\subsubsection{Isozaki's asymptotic representation formula}
For $\tau \geq \abs{\xi}$, put
\begin{equation}
\label{2.16}
S_{j,\tau} := \langle \partial_\nu u_{j,\lambda_\tau^+}^+ ,f_\tau^- \rangle_{L^2(\Gamma)},\ j=1,2,
\end{equation}
and recall from \eqref{vz1}-\eqref{vz2} that $u_{j,\lambda_\tau^+}^+=f_\tau^+ + v_{j,\lambda_\tau^+}^+$ and $v_{j,\lambda_\tau^+}^+=-(A_{q_j}-\lambda_\tau^+)^{-1} (q_j f_\tau^+)$.
Since $v_{j,\lambda_\tau^+}^+ \in D(A_{q_j})$, we have $\partial_\nu u_{j,\lambda_\tau^+}^+ \in L^2(\Gamma)$ from Appendix \ref{sec-L2r}, and hence $S_{j,\tau}$ is well-defined.

The following identity extends the classical Isozaki formula established in \cite{I} for bounded potentials, to possibly
unbounded potentials lying in $\mathcal{Q}_{c_0}(M)$. Its proof can be found in \cite[Proposition 3.1]{BKMS} but, for the sake of completeness and for the convenience of the reader, we provide it below.

\begin{proposition}
\label{p1}
Let $q_j \in \cQ_{c_0}(M)$, $j=1,2$. Then, for all $\xi \in \mathbb{R}^d$, we have
$$\lim_{\tau \rightarrow +\infty} (S_{1,\tau}-S_{2,\tau}) =\int_\Omega (q_1-q_2) e^{-i\xi\cdot x} dx.$$
\end{proposition}
\begin{proof}
For $j=1,2$, we have
\begin{equation} 
\label{20}
 \left\{
\begin{array}{ll}
( -\Delta +{{q_j}}-\lambda_\tau^+)u_{j,\lambda_\tau^+}^+ = 0  & \text{in}\ \Omega\\
u_{j,\lambda_\tau^+}^+ = f_{\tau}^+ & \text{on}\ \Gamma,
 \end{array}
\right. 
\end{equation}
hence by multiplying the first line of \eqref{20} by $\overline{f_\tau^-}$, integrating on $\Omega$ and applying the Green formula, we obtain that
\begin{eqnarray*}
0 &= & \int_\Omega (-\Delta+q_j-\lambda_\tau^+)u^+_{j,\lambda_\tau^+}\overline{f_\tau^- (x)} dx \\
& = & \int_\Gamma f_\tau^+ \overline{ \partial_\nu f_\tau^-} d\sigma-\int_\Gamma (\partial_\nu u_{j,\lambda_{\tau}^+}^+) \overline{ f_\tau^-} d\sigma + \int_\Omega u_{j,\lambda_{\tau}^+}^+ \overline{ (-\Delta+q_j-\lambda_\tau^-)f_\tau^-} dx \\
& = & \int_\Gamma f_\tau^+ \overline{ \partial_\nu f_\tau^-} d\sigma - S_{j,\tau} + \int_\Omega u_{j,\lambda_{\tau}^+}^+ q_j \overline{f_\tau^-} dx,\ j=1,2.
\end{eqnarray*}
Here, we used \eqref{3.36} and \eqref{2.16} in the last line.
Thus, we have
$$
S_{1,\tau}-S_{2,\tau}=\int_\Omega \big(q_1u_{1,\lambda_{\tau}^+}^+ - q_2 u_{2,\lambda_{\tau}^+}^+\big)\overline{f_\tau^-} dx.
$$
This and $u_{j,\lambda_\tau^+}^+=f_\tau^+ + v_{j,\lambda_\tau^+}^+$, $j=1,2$, then yield
$$
S_{1,\tau} -S_{2,\tau} = \int_\Omega (q_1-q_2) f_\tau^+ \overline{f_\tau^-} dx + \int_\Omega q_1 v_{1,\lambda_{\tau}^+}^+ \overline{f_\tau^-} dx - \int_\Omega q_2 
v_{2,\lambda_{\tau}^+}^+ \overline{f_\tau^-} dx.
$$
Taking the limit as $\tau \to +\infty$ in the above identity and using \eqref{l4a}, we get that
\begin{equation}
\label{tm3}
\lim_{\tau\to+\infty} \left( S_{1,\tau}-S_{2,\tau} - \int_\Omega (q_1-q_2) f_\tau^+ \overline{f_\tau^-} dx \right)=0,
\end{equation}
and since $q_1-q_2 \in L^1(\Omega)$, we have
$\lim_{\tau\to+\infty} \int_\Omega (q_1-q_2) f_\tau^+ \overline{f_\tau^-} dx=\int_\Omega (q_1-q_2) e^{-i \xi \cdot x} dx$ 
by \eqref{2} and the dominated convergence theorem. Finally, this and \eqref{tm3} entail \eqref{20}.  
\end{proof}

\section{Proof of Theorem \ref{t1}} 
\label{sec-proof}

Since the stability estimate \eqref{se} is obviously satisfied when $\limsup_{k \to +\infty} \abs{\lambda_{1,k}-\lambda_{2,k}}=+\infty$, we assume without loss of generality in the sequel that $\limsup_{k \to +\infty} \abs{\lambda_{1,k}-\lambda_{2,k}}<+\infty$. As a consequence we have 
\begin{equation}
\label{tm4a}
\sup_{k\geq1} \abs{\lambda_{1,k}-\lambda_{2,k}} \leq C,
\end{equation}
for some positive constant $C$, and hence
\begin{equation}
\label{tm4b}
\abs{\lambda_{2,k}} \leq C(1+\abs{\lambda_{1,k}}),\ k \geq 1,
\end{equation}
upon possibly enlarging $C$. Here and in the remaining part of this proof, $C$ denotes a generic positive constant independent of $k$, which may change from line to line. 

The proof being quite lengthy, we split it into 7 steps.

\noindent {\it Step 1: Introducing an additional spectral parameter.}
We use the same notations as in Section \ref{sec-Isozaki}. Namely, for $z \in \mathbb{C}\setminus [-c_0,+\infty)$, $j=1,2$, we denote by $u^+_{j,z}$ the $W^{2,p}(\Omega)$-solution to the boundary value problem \eqref{tm0}. 
Since $q_j f_\tau^+ \in L^2(\Omega)$ by \eqref{tm1}, we have
$v_{j,z}^+=u_{j,z}^+-f_\tau^+ \in D(A_{q_j})$ from \eqref{vz1}-\eqref{vz2}. Therefore, $\partial_\nu u_{j,z}^+ \in L^2(\Gamma)$ according to Appendix \ref{sec-L2r}, and for all $\mu \in \mathbb{C} \setminus [-c_0,+\infty)$ the normal derivative
of $v^+_{j,\lambda_\tau^+,\mu} := u^+_{j,\lambda_\tau^+}-u^+_{j,\mu}$ lies in $L^2(\Gamma)$. Moreover, we have
\begin{eqnarray}
S_{1,\tau}-S_{2,\tau} 
&= & \langle \partial_\nu u^+_{1,\lambda^+_\tau} -\partial_\nu u^+_{2,\lambda^+_\tau},f_\tau^- \rangle_{L^2(\Gamma)} \nonumber\\
& = & \langle \partial_\nu v^+_{1,\lambda^+_\tau,\mu} ,f_\tau^- \rangle_{L^2(\Gamma)}-\langle \partial_\nu v^+_{2,\lambda^+_\tau,\mu},f_\tau^- \rangle_{L^2(\Gamma)}+\langle \partial_\nu u^+_{1,\mu} -\partial_\nu u^+_{2,\mu},f_\tau^- \rangle_{L^2(\Gamma)}, \label{tm5}
\end{eqnarray}
from \eqref{2.16}. 
We first examine the last term on the right-hand-side of \eqref{tm5}. By H\"older's inequality, we have
$$ \abs{\langle \partial_\nu u^+_{1,\mu} -\partial_\nu u^+_{2,\mu},f_\tau^- \rangle_{L^2(\Gamma)}} \leq \norm{\partial_\nu u^+_{1,\mu} -\partial_\nu u^+_{2,\mu}}_{L^p(\Gamma)} \norm{f_\tau^-}_{L^{p^\prime}(\Gamma)},$$
where $p^\prime:=\frac{2d}{d-2}$ is the H\"older conjugate of $p$.
Thus, we have
$\lim_{\mu \to -\infty} \langle \partial_\nu u^+_{1,\mu} -\partial_\nu u^+_{2,\mu},f_\tau^- \rangle_{L^2(\Gamma)} = 0$ by Lemma \ref{lemma 2.2}, and hence
\begin{equation}
\label{t1c}
S_{1,\tau}-S_{2,\tau}  = 
\lim_{\mu \rightarrow -\infty} \langle \partial_\nu v^+_{1,\lambda^+_\tau,\mu}-\partial_\nu v^+_{2,\lambda^+_\tau,\mu} ,f_\tau^- \rangle_{L^2(\Gamma)},
\end{equation}
from \eqref{tm5}.

\noindent {\it Step 2: Decomposition.} The next step is to apply \eqref{l5a} on the right-hand side of \eqref{t1c}. We get through direct computation that 
\begin{equation}
\label{tm6}
\langle \partial_\nu v^+_{1,\lambda^+_\tau,\mu}-\partial_\nu v^+_{2,\lambda^+_\tau,\mu} ,f_\tau^- \rangle_{L^2(\Gamma)} 
=\sum_{k = 1}^{+\infty} \left( A_k(\mu,\tau)+B_k(\mu,\tau)+C_k(\mu,\tau) \right),
\end{equation}
where
$$A_k(\mu,\tau):=\frac{\mu-\lambda^+_\tau}{(\lambda_\tau^+-\lambda_{1,k})(\mu - \lambda_{1,k})}
\langle f_\tau^+, \psi_{1,k}-\psi_{2,k}\rangle_{L^2(\Gamma)} \overline{\langle f^-_\tau ,\psi_{1,k} \rangle}_{L^2(\Gamma)},$$
$$B_k(\mu,\tau):=\frac{\mu-\lambda^+_\tau}{(\lambda_\tau^+-\lambda_{1,k})(\mu - \lambda_{1,k})} 
\langle f_\tau^+, \psi_{2,k}\rangle_{L^2(\Gamma)} \overline{\langle f^-_\tau ,\psi_{1,k}-\psi_{2,k}\rangle}_{L^2(\Gamma)}$$
and
$$C_k(\mu,\tau):=\left(\frac{\mu-\lambda^+_\tau}{(\lambda_\tau^+-\lambda_{1,k})(\mu - \lambda_{1,k})}-\frac{\mu-\lambda^+_\tau}{(\lambda_\tau^+-\lambda_{2,k})(\mu - \lambda_{2,k})}\right)\langle f_\tau^+, \psi_{2,k}\rangle_{L^2(\Gamma)}\overline{\langle f^-_\tau ,\psi_{2,k} \rangle}_{L^2(\Gamma)}.$$

\noindent {\it Step 3: Majorizing $A_k(\mu,\tau)$ and $B_k(\mu,\tau)$.}  Let us recall from \eqref{tm*} that $\norm{\psi_{j,k}}_{L^2(\Gamma)}\leq C \left(1+\abs{\lambda_{j,k}} \right)$ for $j=1,2$ and all $k \geq 1$, where $C>0$ depends only on $\Omega$ and $q_j$. Thus, with reference to \eqref{3001} with $r=2$ and $X=\Gamma$, we obtain that
$$
\abs{\langle f^\pm_\tau ,\psi_{j,k} \rangle_{L^2(\Gamma)}} \leq C \left(1+\abs{\lambda_{j,k}} \right),\ k \geq 1.
$$
This, \eqref{3001} and \eqref{tm4b} then yield for all $\mu \leq -(1+c)$ and all $\tau \geq 1 + \abs{\xi}$, that
\begin{equation}
\label{tm8}
\abs{A_k(\mu,\tau)} + \abs{B_k(\mu,\tau)} \leq C_\tau \norm{\psi_{1,k}-\psi_{2,k}}_{L^2(\Gamma)},\ k \geq 1.
\end{equation}
Here and below, $C_\tau$ denotes a generic positive constant possibly depending on $\tau$, which is independent of $k$ and $\mu$.

\noindent {\it Step 4: The case of $C_k(\mu,\tau)$.} We turn now to estimating $C_k(\mu,\tau)$. This can be made by rewriting
$\frac{\mu-\lambda^+_\tau}{(\lambda_\tau^+-\lambda_{1,k})(\mu - \lambda_{1,k})}-
\frac{\mu-\lambda^+_\tau}{(\lambda_\tau^+-\lambda_{2,k})(\mu - \lambda_{2,k})}$ as
$\frac{\lambda_{1,k}-\lambda_{2,k}}{(\lambda_\tau^+ - \lambda_{1,k})(\lambda_\tau^+ - \lambda_{2,k})}
- \frac{\lambda_{1,k}-\lambda_{2,k}}{(\mu - \lambda_{1,k})(\mu - \lambda_{2,k})}$ and using \eqref{tm4a}. We obtain that
\begin{equation}
\label{tm9a}
\abs{C_k(\mu,\tau)} \leq C \left( \left| \frac{\lambda_\tau^+-\lambda_{2,k}}{\lambda_\tau^+-\lambda_{1,k}}\right| \Phi_k(\lambda_\tau^+)+\left| \frac{\mu-\lambda_{2,k}}{\mu-\lambda_{1,k}}\right|\Phi_k(\mu) \right), 
\end{equation}
where
\begin{equation}
\label{tm9b}
\Phi_k(z):=
\left| \frac{\langle f^-_\tau ,\psi_{2,k} \rangle_{L^2(\Gamma)}}{z-\lambda_{2,k}} \right| \left| \frac{\langle f^+_\tau ,\psi_{2,k} \rangle_{L^2(\Gamma)}}{z-\lambda_{2,k}} \right|,\ z \in \C \setminus [-c_0,+\infty).
\end{equation}
Notice from \eqref{tm4a} that
\begin{equation}
\label{tm14a}
\left| \frac{\lambda_\tau^+-\lambda_{2,k}}{\lambda_\tau^+-\lambda_{1,k}}\right| \leq 1+C,\ k \ge 1,
\end{equation}
whenever $\tau \ge 1$, and that 
\begin{equation}
\label{tm14b}
\left| \frac{\mu-\lambda_{2,k}}{\mu-\lambda_{1,k}}\right| \leq 1+C,\ k \ge 1,
\end{equation}
provided that $\mu \le -( 1 + c_0)$.
Further, bearing in mind  that $u_{j,z}^\pm$, $j=1,2$, is the solution to the boundary value problem \eqref{tm0}, we find upon multiplying by $\overline{\phi_{2,k}}$ the first 
line of \eqref{tm0} with $j=2$, integrating the result over $\Omega$ and applying the Green formula, that
$$\langle u_{2,z}^\pm , \phi_{j,k} \rangle_{L^2(\Omega)}
=\frac{\langle f_\tau^\pm, \psi_{2,k} \rangle_{L^2(\Gamma)}}{z - \lambda_{2,k}},\ k \geq 1.$$
Thus, 
\begin{equation}
\label{tt1}
\sum_{k=1}^{+\infty} \left| \frac{\langle f^\pm_\tau ,\psi_{2,k} \rangle_{L^2(\Gamma)}}{z-\lambda_{2,k}} \right|^2=
\norm{u_{2,z}^\pm}_{L^2(\Omega)}^2,
\end{equation}
from the Parseval formula.
Moreover, we have 
\begin{equation}
\label{tm9d}
\norm{u_{2,\lambda_\tau^+}^\pm}_{L^2(\Omega)} \leq C,
\end{equation}
whenever $\tau \geq 1$, according to \eqref{3001} with $(r,X)=(2,\Omega)$, and to \eqref{vz1} and \eqref{vz3} with $j=2$. Similarly, since
$u_{2,\mu}^\pm=f_\tau^\pm-(A_{q_2}-\mu)^{-1}(q_2+\lambda_\tau^\pm-\mu) f_\tau^\pm$ from \eqref{vz1}-\eqref{vz2}, we have
\begin{eqnarray*}
\norm{u_{2,\mu}^\pm}_{L^2(\Omega)}
& \leq &\norm{f_\tau^\pm}_{L^2(\Omega)} + 
\frac{\norm{q_2 f_\tau^\pm}_{L^2(\Omega)}+(\abs{\lambda_\tau^\pm} + \abs{\mu})\norm{f_\tau^\pm}_{L^2(\Omega)}}{\abs{\mu+c_0}}
\end{eqnarray*}
for $\mu \leq -(1+c_0)$, and consequently 
$\norm{u_{2,\mu}^\pm}_{L^2(\Omega)} \leq (2+\abs{\lambda_\tau^\pm}+c) \norm{f_\tau^\pm}_{L^2(\Omega)} + 
\norm{q_2 f_\tau^\pm}_{L^2(\Omega)}$. This, \eqref{3001} with $(r,X)=(2,\Omega)$ and
\eqref{tm1} entail that
\begin{equation}
\label{tm9c}
\norm{u_{2,\mu}^\pm}_{L^2(\Omega)} \leq C_\tau,\ \mu \leq (-1+c_0).
\end{equation}

\noindent{\it Step 5: Sending $\mu$ to $-\infty$.} With reference to \eqref{tm9a}, \eqref{tm9b}, \eqref{tm14a}, \eqref{tm14b}, \eqref{tt1}, \eqref{tm9c} and \eqref{tm9c}, and to \eqref{t1a} and \eqref{tm8}, it follows from \eqref{tm6} and the dominated convergence theorem that
\begin{equation}
\label{t1g}
S_{1,\tau}-S_{2,\tau}=\lim_{\mu \rightarrow -\infty} \langle \partial_\nu v^+_{1,\lambda^+_\tau,\mu}-\partial_\nu v^+_{2,\lambda^+_\tau,\mu} ,f_\tau^- \rangle_{L^2(\Gamma)}
=\sum_{k = 1}^{+\infty} \left( A_k^*(\tau)+B_k^*(\tau)+ C_k^* (\tau) \right),
\end{equation}
where
\begin{equation}
\label{tm10}
A_k^*(\tau):=\frac{1}{\lambda_\tau^+-\lambda_{1,k}} 
\langle f_\tau^+, \psi_{1,k}-\psi_{2,k}\rangle_{L^2(\Gamma)} \overline{\langle f^-_\tau ,\psi_{1,k} \rangle}_{L^2(\Gamma)},
\end{equation}
\begin{equation}
\label{tm11}
B_k^*(\tau):=\frac{1}{\lambda_\tau^+-\lambda_{1,k}} \langle f_\tau^+, \psi_{2,k}\rangle_{L^2(\Gamma)} \overline{\langle f^-_\tau ,\psi_{1,k}-\psi_{2,k}\rangle}_{L^2(\Gamma)}
\end{equation}
and
\begin{equation}
\label{tm12}
C_k^*(\tau):=\frac{\lambda_{1,k}-\lambda_{2,k}}{(\lambda_\tau^+-\lambda_{1,k})(\lambda_\tau^+-\lambda_{2,k})}
\langle f_\tau^+, \psi_{2,k} \rangle_{L^2(\Gamma)} \overline{\langle f^-_\tau ,\psi_{2,k} \rangle}_{L^2(\Gamma)}.
\end{equation}

\noindent{\it Step 5: Sending $\tau$ to $+\infty$.} 
Since $\im (\lambda_\tau^+ - \lambda_{j,k})=2\tau$ for all $k \ge 1$, we deduce from \eqref{3001} that
$$
\abs{A_k^*(\tau)} + \abs{B_k^*(\tau)} \leq C \tau^{-1} \norm{\psi_{1,k}-\psi_{2,k}}_{L^2(\Gamma)} \left( \norm{\psi_{1,k}}_{L^2(\Gamma)} + \norm{\psi_{2,k}}_{L^2(\Gamma)} \right)
$$
and
$$
\abs{C_k^*(\tau)} \leq C \tau^{-2} \abs{\lambda_{1,k}-\lambda_{2,k}}  \norm{\psi_{2,k}}_{L^2(\Gamma)}^2,
$$
where the positive constant $C$ is independent of $k$ and $\tau$.
As a consequence we have
$$
\lim_{\tau\to+\infty}A_k^*(\tau)=\lim_{\tau\to+\infty}B_k^*(\tau)=\lim_{\tau\to+\infty}C_k^*(\tau)=0,\ k \ge 1.
$$
This and \eqref{t1g} yield for any natural number $N$, that
\begin{equation}
\label{t1h}
\lim_{\tau\to+\infty} \abs{S_{1,\tau}-S_{2,\tau}} \leq \limsup_{\tau\to+\infty} \sum_{k = N}^{+\infty} \left( \abs{A_k^*(\tau)}+ \abs{B_k^*(\tau)}+ \abs{C_k^*(\tau)} \right).
\end{equation}
On the other hand, applying the Cauchy-Schwarz inequality in \eqref{tm10}, we get from \eqref{tt1} that
\begin{eqnarray}
\sum_{k = N}^{+\infty} 
\abs{A_k^*(\tau)} & \leq & \norm{f_\tau^+}_{L^2(\Omega)} \left( \sum_{k=1}^{+\infty} \left| \frac{\langle f^-_\tau ,\psi_{1,k} \rangle_{L^2(\Gamma)}}{\lambda_\tau^+-\lambda_{1,k}} \right|^2 \right)^{\frac{1}{2}} \left(\sum_{k=N}^{+\infty} \norm{\psi_{1,k}-\psi_{2,k}}_{L^2(\Gamma)}^2\right)^{\frac{1}{2}} \nonumber \\
&\leq &  \norm{f_\tau^+}_{L^2(\Omega)} \norm{u_{1,\lambda_\tau^+}^-}_{L^2(\Omega)}
\left(\sum_{k=N}^{+\infty} \norm{\psi_{1,k}-\psi_{2,k}}_{L^2(\Gamma)}^2\right)^{\frac{1}{2}}.
\label{tm13}
\end{eqnarray}
Further, since $\sup_{\tau \ge 1} \norm{f_\tau^+}_{L^2(\Omega)} \norm{u_{1,\lambda_\tau^+}^-}_{L^2(\Omega)}<+\infty$ according to
 \eqref{3001} with $(r,X)=(2,\Omega)$, \eqref{vz1} and \eqref{vz3}, it follows from \eqref{tm13} that
\begin{equation}
\label{t1i}
\limsup_{\tau\to+\infty}\sum_{k = N}^{+\infty} \abs{A_k^*(\tau)} \leq C \left( \sum_{k=N}^\infty\norm{\psi_{1,k}-\psi_{2,k}}_{L^2(\Gamma)}^2 \right)^{\frac{1}{2}},
\end{equation}
for some constant $C>0$ independent of $N$.

Similarly, by arguing as above with \eqref{tm11} and \eqref{tm12} instead of \eqref{tm10}, we find that 
$$
\limsup_{\tau\to+\infty}\sum_{k = N}^\infty \abs{B_k^*(\tau)}\leq C \left( \sum_{k=N}^{+\infty}\norm{\psi_{1,k}-\psi_{2,k}}_{L^2(\Gamma)}^2 \right)^{\frac{1}{2}}
$$
and
$$
\limsup_{\tau\to+\infty}\sum_{k = N}^{+\infty} \abs{C_k^*(\tau)} \leq C\sup_{k\geq N}\abs{\lambda_{1,k}-\lambda_{2,k}},
$$
which together with \eqref{t1h} and \eqref{t1i}, yields
\begin{equation}
\label{limsup}
\limsup_{\tau\to+\infty} \abs{S_{1,\tau}-S_{2,\tau}}\leq C\left(\sup_{k\geq N}\abs{\lambda_{1,k}-\lambda_{2,k}}+\left( \sum_{k=N}^{+\infty}\norm{\psi_{1,k}-\psi_{2,k}}_{L^2(\Gamma)}^2 \right)^\frac{1}{2} \right).
\end{equation}
Therefore, we have
$$\left| \int_\Omega e^{-ix\cdot\xi}(q_1-q_2)dx \right| \leq C\left(\sup_{k\geq N}\abs{\lambda_{1,k}-\lambda_{2,k}}+\left( \sum_{k=N}^{+\infty}\norm{\psi_{1,k}-\psi_{2,k}}_{L^2(\Gamma)}^2 \right)^\frac{1}{2} \right),
$$
from Proposition \ref{p1}, where $C$ is independent of $N$. Now, with reference to \eqref{t1a}, we find upon sending $N$ to infinity in the above estimate, that
\begin{equation}
\label{t1e}
\left| \int_\Omega e^{-ix\cdot\xi}(q_1-q_2)dx \right| \leq C \limsup_{k\to \infty}\abs{\lambda_{1,k}-\lambda_{2,k}}.
\end{equation}

\noindent{\it Step 7: End of the proof.} Let us denote by $q$ the extension of $q_1-q_2$ by zero in $\mathbb{R}^d \setminus \Omega$, and by $\hat{q}$ the Fourier transform of $q$, i.e., 
$$ \hat{q}(\xi) = \int_{\mathbb{R}^d} e^{-i x \cdot \xi} q(x) dx=\int_\Omega e^{-ix\cdot\xi}(q_1-q_2)dx,\ \xi \in \mathbb{R}^d. $$
Then, setting $\Lambda:=\limsup_{k\to \infty}\abs{\lambda_{1,k}-\lambda_{2,k}}$, we may rewrite \eqref{t1e} as
\begin{equation}
\label{t1j}
\abs{\hat{q}(\xi)} \leq C \Lambda,\ \xi \in \mathbb{R}^d.
\end{equation}
Let $r>0$ be fixed. Putting $B_r:=\{ \xi \in \mathbb{R}^d,\ \abs{\xi} < r \}$ and using that
$\int_{B_r} \abs{\hat{q}(\xi)}^2 d \xi \leq C r^d \norm{\hat{q}}_{L^\infty(B_r)}^2$, we deduce from \eqref{t1j} that
\begin{equation}
\label{t1k}
\int_{B_r} \abs{\hat{q}(\xi)}^2 d \xi \leq C r^d \Lambda^2.
\end{equation}
Next, since
$\int_{\mathbb{R}^d \setminus B_r}  (1+\abs{\xi}^2)^{-1} \abs{\hat{q}(\xi)}^2 d \xi \leq r^{-2} \int_{\mathbb{R}^d \setminus B_r}  \abs{\hat{q}(\xi)}^2 d \xi \leq r^{-2} \int_{\mathbb{R}^d}  \abs{\hat{q}(\xi)}^2 d \xi $, we have
$$
\int_{\mathbb{R}^d \setminus B_r}  (1+\abs{\xi}^2)^{-1} \abs{\hat{q}(\xi)}^2 d \xi 
\leq r^{-2} \norm{q}_{L^2(\mathbb{R}^d)}^2
$$
by Parseval's theorem. Thus, keeping in mind that $\norm{q}_{L^2(\mathbb{R}^d)}=\norm{q}_{L^2(\Omega)}\leq M$, we obtain that
$$ \int_{\mathbb{R}^d \setminus B_r}  (1+\abs{\xi}^2)^{-1} \abs{\hat{q}(\xi)}^2 d \xi \leq M^2 r^{-2}. $$
From this, \eqref{t1k} and the identity $\norm{q}_{H^{-1}(\Omega)}^2 = \int_{\mathbb{R}^d}  (1+\abs{\xi}^2)^{-1} \abs{\hat{q}(\xi)}^2 d \xi$, it then follows that
$$
\norm{q}_{H^{-1}(\Omega)} \leq C \left( r^\frac{d}{2} \Lambda + r^{-1} \right),\ r >0.
$$
Finally, taking $r=\left( \frac{4}{nC} \right)^{\frac{2}{n+2}} \Lambda^{-\frac{2}{n+2}}$ in the above inequality to minimize 
its right-hand side, we get \eqref{se}. This completes the proof of Theorem \ref{t1}.

\begin{appendix}

\section{$L^2(\Gamma)$-regularity of the $\psi_k$'s}
\label{sec-L2r}
For $F\in L^2(\Omega)$ we consider the solution $u \in H_0^1(\Omega)$ to the boundary value problem
\begin{equation} 
\label{22}
 \left\{
\begin{array}{ll}
( -\Delta +q)u=F  &\text{in} \;\Omega\\
    u=0  &\text{on} \;\Gamma,
 \end{array}
\right.
\end{equation}
given by Lax-Milgram's theorem. If $q$ were in $L^\infty(\Omega)$ then
$u$ would be in $H^2(\Omega)$ by elliptic regularity, and satisfy
$$\norm{u}_{H^2(\Omega)}\leq C(\norm{u}_{L^2(\Omega)}+\norm{F}_{L^2(\Omega)}) $$
for some positive constant $C=C(\Omega,\norm{q}_{L^\infty(\Omega)})$.
As a result, we would have $\partial_\nu u \in L^2(\Gamma)$ and the estimate
\begin{equation} 
\label{l3a}
\norm{\partial_\nu  u}_{L^2(\Gamma)} \leq C(\norm{u}_{L^2(\Omega)}+\norm{F}_{L^2(\Omega)}),
\end{equation}
where $C$ is another positive constant depending only on $\Omega$ and $\norm{q}_{L^\infty(\Omega)}$. 
However, since $q$ is possibly unbounded in the framework of this article, we cannot apply the standard theory of elliptic PDEs here,  which leaves us with the task of establishing the following result. 
\begin{proposition}
\label{pr1}
Let $q \in \cQ_{c_0}(M)$, where $c_0>0$ and $M>0$ are fixed, and let $F\in L^2(\Omega)$. Let $u\in H^1(\Omega)$ be a solution to \eqref{22}. Then, we have $\partial_\nu  u \in L^2(\Gamma)$ and the estimate \eqref{l3a} holds for some positive constant $C$ depending only on $\Omega$, $c_0$ and $M$.
\end{proposition}

The derivation of Proposition \ref{pr1} is similar to the one of \cite[Proposition 2.2]{BKMS} but for the sake of self-containedness of this paper and for the convenience of the reader, we provide the proof of this technical result in Appendix \ref{app-B}, below.

Notice that it follows from Proposition \ref{pr1} that any function $u \in D(A_q)$ has a normal derivative $\partial_\nu u \in L^2(\Gamma)$ satisfying
$$ 
\norm{\partial_\nu u}_{L^2(\Gamma)}\leq C \left( \norm{u}_{L^2(\Omega)}+\norm{A_q u}_{L^2(\Omega)} \right),
$$
where $C$ is a positive constant depending only on $\Omega$ and $M$. 
Specifically, since all the eigenfunctions $\phi_k$, $k \geq 1$, lie in $D(A_q)$, we get that $\psi_k \in L^2(\Gamma)$ and that
\begin{equation}
\label{tm*} 
\norm{\psi_k}_{L^2(\Gamma)} \leq C(1+\abs{\lambda_k}).
\end{equation}

\section{Proof of Proposition \ref{pr1}}
\label{app-B}

By linearity of \eqref{22}, we may assume without limiting the generality of the foregoing that $F$ is real-valued. As a consequence, the solution $u$ to \eqref{22} is real-valued as well.

Put $q_0:=q+c_0$. Since $q_0 \ge 0$ by assumption, we pick a sequence $(q_\ell)_{\ell \ge 1} \in \mathcal C^\infty(\overline{\Omega})$ of non-negative functions satisfying
\begin{equation}
\label{eq-s0}
\lim_{\ell \to +\infty} \norm{q_\ell-q_0}_{L^{\frac{3d}{5}}(\Omega)} =0.
\end{equation}
Then, for each $\ell \ge 1$ we consider the solution $u_\ell \in H^2(\Omega)\cap H^1_0(\Omega)$ to the boundary value problem
\begin{equation} \label{222}
 \left\{
\begin{array}{ll}
( -\Delta +q_\ell)u_\ell=c_0 u+F  &\text{in} \;\Omega\\
    u_\ell=0  &\text{on} \;\Gamma.\\
 \end{array}
\right.
\end{equation}

We split the proof into 4 steps.

\noindent{\it Step 1: The sequence $(u_\ell)_{\ell\ge 1}$ is bounded in $H^1(\Omega)$.} For $\ell \in \mathbb N$ fixed, we multiply the first equation of \eqref{222} by $u_\ell$ and integrate over $\Omega$. We obtain
$\int_\Omega \abs{\nabla u_\ell}^2 dx + \int_\Omega q_\ell \abs{u_\ell}^2 dx =\int_\Omega G u_\ell dx$ with the help of
Green's formula, where $G:=c_0 u+F$.  As a consequence, we have
$$
\int_\Omega \abs{\nabla u_\ell}^2 dx + \int_\Omega q_0 \abs{u_\ell}^2 dx  =\int_\Omega G u_\ell dx- \int_\Omega (q_\ell-q_0) \abs{u_\ell}^2 dx,
$$
which, upon applying Poincaré's inequality and remembering that $q_0 \ge 0$ a.e. in $\Omega$, leads to
$$\norm{u_\ell}^2_{H^1(\Omega)} \leq  C_0 \left( \int_\Omega \abs{q_\ell-q_0} \abs{u_\ell}^2 dx +\int_\Omega \abs{G}  \abs{u_\ell} dx \right). $$
Here and in the sequel, $C_0$ denotes a generic positive constant, depending only on $\Omega$. Taking into account that $H^1(\Omega) \subset L^{\frac{2d}{d-2}}(\Omega)$ and that the embedding is continuous, by the Sobolev embedding theorem (see e.g. \cite[Theorem 1.4.4.1]{Gr}), we infer from the above inequality and Hölder's inequality that for all $\epsilon>0$,
\begin{eqnarray}
\norm{u_\ell}^2_{H^1(\Omega)} 
& \leq  & C_0 \left( \norm{q_\ell-q_0}_{L^{\frac{3d}{5}}(\Omega)} \norm{u_\ell}^2_{L^{\frac{6d}{3d-5}}(\Omega)}+ \int_\Omega \abs{G}  \abs{u_\ell} dx  \right) \nonumber \\
& \leq  & C_0 \left(  \norm{q_\ell-q_0}_{L^{\frac{3d}{5}}(\Omega)} \norm{u_\ell}^2_{H^1(\Omega)}+\epsilon \norm{u_\ell}^2_{L^2(\Omega)}+\epsilon^{-1} \norm{G }_{L^2(\Omega)}^2  \right).\label{eq-s1.2}
\end{eqnarray}
Now, with reference to \eqref{eq-s0}, we pick $\ell_0 \geq 1$ so large
$\norm{q_\ell-q_0}_{L^{\frac{3d}{5}}(\Omega)}\leq \epsilon$ for all $\ell\geq\ell_0$. From this and \eqref{eq-s1.2} it then follows that
$\norm{u_\ell}^2_{H^1(\Omega)} \leq C_0 \left( \epsilon \norm{u_\ell}^2_{H^1(\Omega)} + \epsilon^{-1} \norm{G}_{L^2(\Omega)}^2 \right)$ whenever $\ell \geq \ell_0$. Thus, by taking $\epsilon=(2C_0)^{-1}$ in this estimate, we find that
\begin{equation} 
\label{l3b}
\norm{u_\ell}_{H^1(\Omega)} \leq C_0 \left( \norm{u}_{L^2(\Omega)} + \norm{F}_{L^2(\Omega)} \right),\ \ell \geq \ell_0.
\end{equation}

\noindent{\it Step 2: $(u_\ell)_{\ell \geq 1}$ converges to $u$ in $W^{2,p}(\Omega)$, where $p:=2d/(d+2)$}. For $\ell\ \geq \ell_0$ fixed, we put $v_\ell :=u-u_\ell$ in such a way that
\begin{equation}
\label{eq-s2.1}
 \left\{
\begin{array}{ll}
 -\Delta v_\ell +q_0v_\ell=(q_\ell-q_0)u_\ell  &\text{in} \;\Omega\\
    v_\ell=0  &\text{on} \;\Gamma,
 \end{array}
\right.
\end{equation}
according to \eqref{22} and \eqref{222}.
Thus, bearing in mind that $q_0\geq0$ a.e. in $\Omega$, we have
\begin{equation}
\label{l3bb}
\norm{v_\ell}_{H^1(\Omega)} \leq C_0 \norm{(q_\ell-q_0)u_\ell}_{H^{-1}(\Omega)}.
\end{equation}
Moreover, $H^1_0(\Omega)$ being continuously embedded in $L^{\frac{2d}{d-2}}(\Omega)$, the space $L^p(\Omega)$ is, by duality, continuously embedded in $H^{-1}(\Omega)$, and \eqref{l3bb} yields
\begin{equation} 
\label{l3cc}
\norm{v_\ell}_{H^1(\Omega)}\leq C_0\norm{(q_\ell-q_0)u_\ell}_{L^p(\Omega)}.
\end{equation}
Further, in light of \eqref{eq-s2.1} we have
$\norm{v_\ell}_{W^{2,p}(\Omega)}\leq C_0 \left( \norm{q_0v_\ell}_{L^p(\Omega)}+\norm{(q_\ell-q_0)u_\ell}_{L^p(\Omega)} \right)$
from \cite[Theorems 2.4.2.5]{Gr}, and
$\norm{q_0v_\ell}_{L^p(\Omega)}\leq \norm{q_0}_{L^{\frac{d}{2}}(\Omega)}\norm{v_\ell}_{L^{\frac{2d}{d-2}}(\Omega)}\leq C\norm{v_\ell}_{H^1(\Omega)},$
by H\"older's inequality and the Sobolev embedding theorem, where, from now on, $C$ is a generic positive constant depending only on $\Omega$, $c_0$ and $M$.
From this and \eqref{l3cc} it then follows that
\begin{eqnarray*}
\norm{v_\ell}_{W^{2,p}(\Omega)} & \leq & C \norm{(q_\ell-q_0)u_\ell}_{L^p(\Omega)} \nonumber \\
& \le & C \norm{q_\ell-q_0}_{L^{\frac{3d}{5}}(\Omega)} \norm{u_\ell}_{H^1(\Omega)},
\end{eqnarray*}
which together with 
\eqref{eq-s0} and \eqref{l3b}, yields
\begin{equation}
\label{eq2}
\lim_{\ell \to +\infty} \norm{u_\ell-u}_{W^{2,p}(\Omega)}=0.
\end{equation}

\noindent{\it Step 3: The sequence $(u_\ell)_{\ell \ge 1}$ is bounded in $W^{2,r}(\Omega)$, where $r:=6d \slash(3d+4)$}. In light of \eqref{222} and the identity $G=cu +F$, we infer from \cite[Theorem 2.4.2.5]{Gr} upon taking into account that $L^2(\Omega)$ is continuously embedded in $L^{r}(\Omega)$, that
\begin{equation} 
\label{l3c}
\norm{u_\ell}_{W^{2,r}(\Omega)}\leq C (\norm{q_\ell u_\ell}_{L^{r}(\Omega)}+\norm{G}_{L^2(\Omega)}),\ \ell \ge 1.
\end{equation}
Further, as we have
$\norm{q_\ell u_\ell}_{L^{r}(\Omega)}\leq \norm{q_\ell}_{L^{\frac{3d}{5}}(\Omega)}\norm{u_\ell}_{L^{\frac{2d}{d-2}}(\Omega)}\leq C \norm{q_\ell}_{L^{\frac{3d}{5}}(\Omega)}\norm{u_\ell}_{H^1(\Omega)}$
by H\"older's inequality and the Sobolev embedding theorem, and hence
$$\norm{q_\ell u_\ell}_{L^{r}(\Omega)}\leq C(\norm{u}_{L^2(\Omega)}+\norm{F}_{L^2(\Omega)}),\ \ell \ge \ell_0, $$
from \eqref{eq-s0} and \eqref{l3b}, it follows from \eqref{l3c} that
\begin{equation} 
\label{l3d}
\norm{u_\ell}_{W^{2,r}(\Omega)}\leq C \left( \norm{u}_{L^2(\Omega)}+\norm{F}_{L^2(\Omega)} \right),\ \ell \ge \ell_0.
\end{equation}

\noindent{\it Step 4: End of the proof}. We are left with the task of establishing that
\begin{equation}
\label{eq-s4.1}
\norm{\partial_\nu u_\ell}_{L^2(\Gamma)} \leq C \left(\norm{u}_{L^2(\Omega)}+\norm{F}_{L^2(\Omega)} \right),\ \ell \ge \ell_0.
\end{equation}
For this end we consider a vector field $\gamma \in \mathcal{C}^1(\overline{\Omega},\mathbb{R}^d)$ satisfying $\gamma_{|_{\Gamma}}=\nu$, multiply the first line of \eqref{222} by $\gamma\cdot \nabla u_\ell$, and integrate over $\Omega$. We obtain that
\begin{equation}
\label{145}
\int_\Omega (-\Delta u_\ell) \gamma\cdot \nabla u_\ell dx + \int_\Omega q_\ell\; u_\ell \gamma\cdot \nabla u_\ell dx =\int_\Omega c\;u \gamma \cdot \nabla u_\ell dx+\int_\Omega F \gamma \cdot \nabla u_\ell dx,\ \ell \ge 1.
\end{equation}
Applying the divergence formula, the first term on the left-hand side of \eqref{145} reads
\begin{equation}
\label{eq-s4.2}
\int_\Omega (\Delta u_\ell) \gamma\cdot \nabla u_\ell dx 
= \int_\Gamma | \partial_\nu u_\ell |^2 d\sigma -\int_\Omega \nabla(\gamma\cdot \nabla u_\ell)\cdot \nabla u_\ell dx,\ \ell \ge 1.
\end{equation}
Further, writing $\gamma=(\gamma_1,\ldots,\gamma_d)^T$, we get through direct computation that
\begin{eqnarray}
\nabla(\gamma\cdot \nabla u_\ell)\cdot \nabla u_\ell &= & \sum_{i,j= 1}^d \left( \partial_i \left( \gamma_j \partial_j u_\ell \right) \right) \partial_i u_\ell 
\nonumber \\
&= & \sum_{i,j=1}^d (\partial_i \gamma_j) (\partial_j u_\ell) \partial_i u_\ell +\frac{1}{2} \gamma\cdot\nabla \abs{\nabla u_\ell}^2,\ 
\ell \ge 1. \label{eq-s4.3}
\end{eqnarray}
Next, since $\gamma \cdot \nu=1$ on $\Gamma$ and $\abs{\nabla u_\ell}=\abs{\partial_\nu u_\ell}$ on $\Gamma$, we have 
$$\int_\Omega \gamma \cdot\nabla \abs{\nabla u_\ell}^2 dx = \norm{\partial_\nu u_\ell}_{L^2(\Gamma)}^2 -\int_\Omega  ( \nabla \cdot \gamma) \abs{\nabla u_\ell}^2 dx, $$
and hence
$$\int_\Omega \Delta u_\ell (\gamma\cdot \nabla u_\ell) dx =\frac{1}{2} \norm{\partial_\nu u_\ell}_{L^2(\Gamma)}^2 +\int_\Omega H(x) \nabla u_\ell(x) dx,\ \ell \ge 1. $$
from \eqref{eq-s4.2}-\eqref{eq-s4.3}, where
$$H(x) X :=-\sum_{i,j=1}^d (\partial_i \gamma_j)(x) X_j X_i + \frac{1}{2} \left( \nabla \cdot \gamma(x) \right) \abs{X}^2,\ X=(X_1,\ldots,X_d)\in\R^d,\ x \in \Omega. $$
It follows readily from this and \eqref{145} that
\begin{equation}
\label{147}
\frac{1}{2} \norm{\partial_\nu u_\ell}_{L^2(\Gamma)}^2=-\int_\Omega H(x) \nabla u_\ell(x) dx+ \int_\Omega q_\ell u_\ell\ \gamma\cdot \nabla u_\ell dx -\int_\Omega G\ \gamma \cdot \nabla u_\ell dx,\ \ell \ge 1.
\end{equation}
The second term on the right hand side of \eqref{147} can be bounded with the help of H\"older's inequality, as
\begin{eqnarray*}
\left| \int_\Omega q_\ell\; u_\ell\ \gamma\cdot \nabla u_\ell dx \right| &\leq & \norm{\gamma}_{L^\infty(\Omega)^d}\norm{q_\ell}_{L^{\frac{3d}{5}}(\Omega)}\norm{u_\ell}_{L^{\frac{6d}{3d-8}}(\Omega)}\norm{\nabla u_\ell}_{L^{\frac{6d}{3d-2}}(\Omega)}\\
&\leq & C\norm{q_\ell}_{L^{\frac{3d}{5}}(\Omega)}\norm{u_\ell}_{L^{\frac{6d}{3d-8}}(\Omega)}\norm{ u_\ell}_{W^{1,\frac{6d}{3d-2}}(\Omega)},\end{eqnarray*}
in such a way that we have
$\left|\int_\Omega q_\ell\ u_\ell\ \gamma\cdot \nabla u_\ell dx\right| \leq C \norm{q_\ell}_{L^{\frac{3d}{5}}(\Omega)}\norm{u_\ell}_{W^{2,r}(\Omega)}^2$
by the Sobolev embedding theorem, and consequently 
$$\left|\int_\Omega q_\ell\; u_\ell\ \gamma\cdot \nabla u_\ell dx\right|\leq C\left(\norm{u}_{L^2(\Omega)}+\norm{F}_{L^2(\Omega)}\right)^2,\ \ell \ge \ell_0$$
from \eqref{l3d}. 
Putting this together with \eqref{l3b} and \eqref{147}, we obtain \eqref{eq-s4.1}.

As a consequence, the sequence $(\partial_\nu u_\ell)_{\ell \ge 1}$ is weakly convergent in $L^2(\Gamma)$, by Banach-Alaoglu's theorem, and we denote by $w$ its weak limit in $L^2(\Gamma)$. On the other hand, since $(u_\ell)_{\ell \ge 1}$ converges to $u$ in the norm-topology of $W^{2,p}(\Omega)$ according to \eqref{eq2}, $(\partial_\nu u_\ell)_{\ell \ge 1}$ strongly converges to $\partial_\nu u$ in $L^p(\Gamma)$. Therefore, we have $\partial_\nu u=w\in L^2(\Gamma)$ by uniqueness of the limit, which proves the first claim of Proposition \ref{pr1}. Finally, \eqref{l3a} follows from \eqref{eq-s4.1} together with the weak convergence in $L^2(\Gamma)$ of $(\partial_\nu u_\ell)_{\ell \ge 1}$ to $\partial_\nu u$.

\section{Influence of potential and spectral parameter on the Neumann response}

Let $q \in \mathcal{Q}_{c_0}(M)$, where $c_0>0$ and $M>0$ are fixed, and pick $\lambda \in \C \setminus [-c_0,+\infty)$ in such a way that $\lambda$ lies in the resolvent set of $A_q$. Then, for all $f \in H^{3/2}(\Gamma)$, we recall from \cite[Lemma 2.3 and Corollary 2.4]{P} that the boundary value problem
 \begin{equation}
 \label{1}  
 \left\{  
\begin{array}{ll}
 (-\Delta +q-\lambda)u=0 &\text{in}\; \Omega\\
   u=f &\text{on}\;\Gamma, \\          
 \end{array}                     
\right. 
\end{equation}
admits a unique solution $u_\lambda \in W^{2,p}(\Omega)$, where we recall that $p= \frac{2d}{d+2}$, satisfying
\begin{equation}\label{u}
\| u_\lambda \|_{W^{2,p}(\Omega)} \leq C_\lambda \|f\|_{H^{3/2}(\Gamma)},
\end{equation}
for some positive constant $C_\lambda$ depending on $\lambda$. Evidently, $u_\lambda$ depends on $q$ as well, but to ease the notation we suppress this dependence. 

Next, since the trace operator $v \longmapsto \partial_\nu v$ sends $W^{2,p}(\Omega)$ into $W^{1-\frac{1}{p},p}(\Gamma)$, we see that $\partial_\nu u_\lambda$ is well-defined in $L^p(\Gamma)$. We first examine the dependence of $\partial_\nu u_\lambda$ with respect to $q$ in the asymptotic regime $\lambda \to -\infty$.  The following result, which is borrowed from \cite[Lemma 2.1]{BKMS}, specifies that the
influence of the potential $q$ on $\partial_\nu u_\lambda$ is, in some sense, dimmed as $\lambda$ goes to $-\infty$.

\begin{lemma}
\label{lemma 2.2}
Let $q_j \in \cQ_{c_0}(M)$, $j=1,2$. For $\lambda \in \R \setminus [-c_0,+\infty)$, denote by $u_{j,\lambda}$ the solution to the boundary value problem \eqref{1} with $q=q_j$. Then, we have
$$
\lim_{\lambda \to -\infty}\| \partial_\nu u_{1,\lambda}- \partial_\nu u_{2,\lambda} \|_{L^p(\Gamma)} = 0.
$$
\end{lemma}

Having seen this, we seek to examine the dependence of $\partial_\nu u_\lambda$ with respect to the spectral parameter $\lambda$ when the potential $q$ is fixed. This can be done with the aid of \cite[Lemma 2.2]{BKMS}, which expresses the difference $\partial_\nu (u_\lambda-u_\mu)$ in terms of $\lambda$, $\mu \in \C \setminus [-c_0,+\infty)$ and $\bsd(A_q)$, as
\begin{equation}
\label{l5a}
\partial_\nu (u_\lambda - u_\mu)=(\mu - \lambda) \sum_{k=1}^{+\infty}\frac{\langle f, \psi_k \rangle_{L^2(\Gamma)}}{(\lambda-\lambda_k)(\mu-\lambda_k)} \psi_k,
\end{equation}
the series on the right-hand-side of \eqref{l5a} being convergent in $L^2(\Gamma)$.

\end{appendix}

\bigskip

\paragraph{\bf Acknowledgement}
The two authors of this article are partially supported by the Agence Nationale de la Recherche under grant ANR-17- CE40-0029.


\begin{thebibliography}{99}
\bibitem{ASS} {\sc G. Alessandrini, J. Sylvester and Z. Sun}, {\em Stability for multidimensional inverse spectral problem}, Commun. Partial Diff. Eqns., \textbf{15} (5) (1990), 711-736.
 \bibitem{A}{\sc  V. A. Ambartsumian}, {\em \"Uber eine Frage der Eigenwerttheorie}, Z. Phys., \textbf{53} (1929), 690-695.
\bibitem{Bo}{\sc G. Borg}, {\em Eine Umkehrung der Sturm-Liouvilleschen Eigenwertaufgabe}, Acta Math.,
\textbf{78} (1946), 1-96. 
\bibitem{BCFKS}{\sc M. Bellassoued, M. Choulli, D. D. Feirrera, Y. Kian, P. Stefanov}, {\em A Borg-Levinson theorem for magnetic Schr\"odinger operators on a Riemannian manifold}, to appear in Annales de l'Institut Fourier, arXiv:1807.08857.
\bibitem{BCY}{\sc M. Bellassoued, M. Choulli, M. Yamamoto}, {\em Stability estimate for an inverse wave equation and a multidimensional Borg-Levinson theorem}, J. Diff. Equat., \textbf{247}(2)  (2009), 465-494.
\bibitem{BKMS} {\sc M. Bellassoued, Y. Kian, Y. Mannoubi and \'E. Soccorsi}, {\em Identification of unbounded electric potentials through asymptotic boundary spectral data}, arXiv:2202.03694.
\bibitem{CK1} {\sc B. Canuto and O. Kavian}, {\em Determining Coefficients in a Class of Heat Equations via Boundary Measurements}, SIAM Journal on Mathematical Analysis, \textbf{32} no. 5 (2001), 963-986.
\bibitem{CK2} {\sc B. Canuto and O. Kavian}, {\em Determining Two Coefficients in Elliptic Operators via Boundary Spectral Data: a Uniqueness Result}, Bolletino Unione Mat. Ital. Sez. B Artic. Ric. Mat. (8), \textbf{7} no. 1 (2004), 207-230.
\bibitem{CS}{\sc M. Choulli and P. Stefanov}, {\em Stability for the multi-dimensional Borg-Levinson theorem with partial spectral data}, 
Commun. Partial Diff. Eqns., \textbf{38} (3) (2013), 455-476.
\bibitem{GL} {\sc I. M. Gel'fand and B. M. Levitan}, {\em On the determination of a differential equation from its spectral function}, Izv. Akad. Nauk USSR, Ser. Mat., \textbf{15} (1951), 309-360.
\bibitem{Gr}  Grisvard. P., Elliptic Problems in Nonsmooth Domains, Pitman, 1985.
\bibitem{I} {\sc H. Isozaki}, {\em Some remarks on the multi-dimensional Borg-Levinson theorem},
J. Math. Kyoto Univ., \textbf{31} (3) (1991), 743-753.
\bibitem{KKS}{\sc O. Kavian, Y. Kian, E. Soccorsi}, {\it Uniqueness and stability results for an inverse spectral problem in a periodic waveguide}, J. Math. Pures Appl., \textbf{104} (2015), no. 6, 1160-1189.
\bibitem{Ki1}{\sc Y. Kian}, {\em  A multidimensional Borg-Levinson theorem for magnetic Schr\"odinger operators with partial spectral data}, Journal of Spectral Theory, \textbf{8} (2018), 235-269.
\bibitem{L}{\sc N. Levinson}, {\em The inverse Sturm-Liouville problem}, Mat. Tidsskr. B., \textbf{1949} (1949), 25-30.
\bibitem{LM1}{\sc J.-L. Lions and E. Magenes}, 
{\em Non-homogeneous Boundary Value Problems and Applications}, Vol. I, 
Spring
er-Verlag, Berlin, 1972.
\bibitem{NSU} {\sc A. Nachman, J. Sylvester, G. Uhlmann}, {\em An n-dimensional Borg-Levinson theorem}, Comm. Math. Phys., \textbf{115} (4) (1988), 595-605.
\bibitem{PS} {\sc L. P\"aiv\"arinta, V. Serov}, {\em An n-dimensional Borg-Levinson theorem for singular
potentials}, Adv. in Appl. Math., \textbf{29} (2002), no. 4, 509-520.
\bibitem{P} {\sc V. Pohjola}, \emph{Multidimensional Borg-Levinson theorems for unbounded potentials}, Asymptot. Anal., \textbf{110} (2018), no. 3-4, 203-226.
\bibitem{S} {\sc \'E. Soccorsi}, \emph{Multidimensional Borg-Levinson inverse spectral problems}, Contemp. Math., \textbf{757} (2020), 19-50.
\end{thebibliography}
\end{document}